\documentclass[12pt,reqno]{article}

\usepackage[usenames]{color}
\usepackage{amssymb}
\usepackage{graphicx}
\usepackage{amscd}

\usepackage[colorlinks=true,
linkcolor=webgreen,
filecolor=webbrown,
citecolor=webgreen]{hyperref}

\definecolor{webgreen}{rgb}{0,.5,0}
\definecolor{webbrown}{rgb}{.6,0,0}

\usepackage{color}
\usepackage{fullpage}
\usepackage{float}

\usepackage{graphics,amsmath}
\usepackage{amsthm}
\usepackage{amsfonts}
\usepackage{latexsym}

\newcommand{\seqnum}[1]{\href{http://oeis.org/#1}{\underline{#1}}}

\begin{document}

\theoremstyle{plain}
\newtheorem{theorem}{Theorem}
\newtheorem{corollary}[theorem]{Corollary}
\newtheorem{lemma}[theorem]{Lemma}
\newtheorem{proposition}[theorem]{Proposition}

\theoremstyle{definition}
\newtheorem{definition}[theorem]{Definition}
\newtheorem{example}[theorem]{Example}
\newtheorem{conjecture}[theorem]{Conjecture}

\theoremstyle{remark}
\newtheorem{remark}[theorem]{Remark}

\newcommand{\lrf}[1]{\left\lfloor #1\right\rfloor}

\begin{center}
\vskip 1cm{\LARGE\bf Reciprocal series involving Horadam numbers}

\vskip 1cm

\Large Kunle Adegoke \\ Department of Physics and Engineering Physics \\ Obafemi Awolowo University \\
220005 Ile-Ife \\ Nigeria \\
\href{mailto:adegoke00@gmail.com}{\tt adegoke00@gmail.com} \\
\vskip 1.0cm
\Large Robert Frontczak\footnote{Statements and conclusions made in this article are entirely those of the author. They do not necessarily reflect the views of LBBW.} \\
Landesbank Baden-W\"{u}rttemberg (LBBW) \\ 70173 Stuttgart \\ Germany \\
\href{mailto:robert.frontczak@web.de}{\tt robert.frontczak@web.de} \\
\vskip 1.0cm
\Large Taras Goy \\ Faculty of Mathematics and Computer Science \\ Vasyl Stefanyk Precarpathian National University \\
76018 Ivano-Frankivsk \\ Ukraine \\
\href{mailto:taras.goy@pnu.edu.ua}{\tt taras.goy@pnu.edu.ua} \\
\end{center}

\vskip .2 in

\begin{abstract}
We evaluate some new three parameter families of finite reciprocal sums involving Horadam numbers. 
We will also be able to state the results for the infinite sums. 
Some Fibonacci and Lucas sums will be presented as examples.
\end{abstract}

\section{Introduction and motivation}

The Horadam sequence $(w_n)_{n\in\mathbb{Z}} = (w_n(a,b;p,q))_{n\in\mathbb{Z}}$ is defined \cite{Horadam1965} recursively by

\begin{equation*}\label{Horadam-def}
w_0 = a, \quad w_1 = b,\qquad w_n  = pw_{n - 1} - qw_{n - 2},\quad n \ge 2,
\end{equation*}

where $a$, $b$, $p$, and $q$ are arbitrary (possibly complex) numbers. 
The sequences $u_n(p,q)=w_n(0,1;p,q)$ and $v_n(p,q)=w_n(2,p;p,q)$ are called Lucas sequences of the first kind and of the second kind, respectively. The most well-known Lucas sequences are the Fibonacci numbers $F_n=u_n(1,-1)$, the Lucas numbers $L_n=v_n(1,-1)$, the Pell numbers $P_n=u_n(2,-1)$, the Pell-Lucas numbers $Q_n=v_n(2,-1)$, the balancing numbers $B_n=u_n(6,1)$, and some others.
All sequences are indexed in the On-Line Encyclopedia of Integer Sequences \cite{OEIS}.

Denote by $\alpha$ and $\beta$, with $|\alpha|>|\beta|$, the distinct roots of the characteristic equation $x^2-px+q=0$ having discriminant 
$\Delta=p^2-4q \neq 0$. The Binet formulas for $w_n$, $u_n$, and $v_n$ are given by

\begin{equation}\label{Binet}
w_n = \frac{A\alpha^n - B\beta^n}{{\alpha  - \beta }}\,,\qquad
u_n = \frac{\alpha^n-\beta^n}{\alpha-\beta}\,,\qquad v_n = \alpha^n+\beta^n\,,
\end{equation}

where $A=b-a\beta$ and $B=b-a\alpha$. Note that $\alpha > 0$, if $p>0$ and $q\leq p^2/4$ or $p\leq 0$ and $q<0$. 
Similarly, $\beta < 0$, if $p\geq 0$ and $q<0$ or $p<0$ and $q\leq p^2/4$. We will also need an expression for negatively 
subscripted Horadam numbers. For negative subscripts the sequences are given by

\begin{equation}\label{hor_neg}
w_{-n} = \frac{av_n-w_n}{q^n}, \qquad u_{-n}=-u_nq^{-n}, \qquad v_{-n}=v_nq^{-n}\,.
\end{equation}


We require the following identity \cite[Formula (4.1)]{Horadam1987} in the sequel  

\begin{equation}\label{eq.emr6a3s}
w_{n}w_{n+r+s} - w_{n+r} w_{n+s}=e_w q^n u_r u_s\,,
\end{equation}
where $e_w=-AB=pab-qa^2-b^2$. For Fibonacci  numbers $e_F=-1$, while for Lucas numbers $e_L=5$.  

The goal of this study is to evaluate a range of finite and infinite families of reciprocal sums involving the Horadam sequence. 
The interest in evaluating Fibonacci and Lucas (re\-lated) reciprocal sums in closed form is not new. 
The topic challenges the mathematical community for decades. In 1974, Miller \cite{Miller} proposed the problem of proving that

\begin{equation}\label{Miller}
\sum_{i=0}^\infty \frac{1}{F_{2^i}} = \frac{7-\sqrt5}{2}.
\end{equation} 

Miller's proposal stimulated great interest in the series of reciprocal Fibonacci numbers, which led to the many proofs and generalizations 
(see survey paper \cite{Duverney} and \cite{Bruckman,Greig,Hoggatt2,Horadam1988} for more information and references).  
Note that in \cite{Miller} the author's name of the problem is indicated incorrectly as Millin (see \cite{Shannon}). 

In 1974, Good \cite{Good} showed that
\begin{equation*}\label{Good}
\sum_{i=0}^N \frac{1}{F_{2^i}} = 3 - \frac{F_{2^N-1}}{F_{2^N}}.
\end{equation*} 
Allowing $N$ to approach infinity, we have \eqref{Miller}. Hoggatt and Bicknell \cite{Hoggatt1} gave eleven methods for finding the value of the sum \cite{Miller}. Shortly later, in \cite{Hoggatt2} they proved a more general formula

\begin{equation*}
\sum_{i=0}^\infty \frac{1}{F_{k2^i}} = \frac{1}{F_k}+\frac{\Phi^2+1}{\Phi(\Phi^{2k}+1)},
\end{equation*} 

where $\Phi=(1+\sqrt5)/2$ is the golden ratio. 	In 1990, Andr\'{e}-Jeannin \cite[Theorem 2]{A-Jeannin1990} expressed
the infinite reciprocal series

\begin{equation*}
\sum_{i=1}^\infty \frac{1}{u_{ki}u_{k(i+1)}} \qquad \mbox{and} \qquad \sum_{i=1}^\infty \frac{1}{v_{ki}v_{k(i+1)}},
\end{equation*}

with odd parameter $k$ in terms of Lambert series $\sum_{n=1}^{\infty}\frac{x^n}{1-x^n}$, $|x|<1$.
Melham \cite{Melham-Shannon} considered the analogues of sequences $u_n$ and $v_n$ for the recurrence $w_n=pw_{n-1}-w_{n-2}$, 
and obtained analogues of Andre-Jeannin's results for these sequences. In 1997, Andr\'{e}-Jeannin \cite[Theorem 2$'$]{A-Jeannin1997} 
again studied the reciprocals of second-order recurrences and evaluated the series

\begin{equation*}
\sum_{i=1}^\infty \frac{q^{ni}}{w_{ni+m}w_{n(i+k)+m}} \qquad \mbox{and} \qquad \sum_{i=1}^\infty \frac{1}{w_{ni+m}w_{n(i+k)+m}},
\end{equation*}

for integers $m\geq0$, $n\geq1$, and $k\geq1$. Some years later, Hu et al. \cite[Theorem 1]{Hu} obtained a general result, 
which contains the evaluation of the finite (and infinite) series

\begin{equation*}
\sum_{i=1}^{N-1} \frac{q^{ni}}{w_{ni+m}w_{n(i+1)+m}},
\end{equation*}

as a special case. In \cite{Laohakosol}, Laohakosol and Kuhapatanakul extended this result to reciprocal sums of second order recurrence sequences with non-constant coefficients.  

Other types of Fibonacci and Lucas (related) reciprocal series, both finite or infinite and alternating or non-alternating,  
are studied in \cite{Adegoke,Bruckman,Farhi,Frontczak-NNTDM2,Frontczak-AMS,Horadam1988,Kuhapatanakul,Melham2012-Integers,Melham2015,Melham-JIS,Omur,Popov,Rabinowitz1,Rabinowitz2,Rayaguru,Sanford}, among others. Focusing on reciprocal sums with three and more factors we refer  
to \cite{Frontczak-NNTDM1,Kilic-Ersanli,Kilic-Prodinger,Melham2014-JIS,Melham2013,Melham2014,Melham2002-FQ,Melham2000,Melham2001}. 

The series that are studied in the present paper are three-parameter series of the form
\begin{equation*}
\sum_{i=1}^N \frac{q^{m(i-k)}}{w_{m(i-k)+n}w_{m(i+k)+n}} \qquad \mbox{and} \qquad 
\sum_{i = 1}^N \frac{q^{m(2i - k)}}{w_{m(2i - k)} w_{m(2i + k)}},
\end{equation*}
where $m$, $k$, and $n$ are integers. To the best of our knowledge, such types of Horadam reciprocal series have not been under consideration yet. For all series we provide closed forms in the finite and infinite cases using an elementary approach. 

We require the following telescoping summation identities with any integers $N$ and $t$:
\begin{equation}\label{Telescop1}
\sum_{i = 1}^N \big( f(i + t) - f(i) \big)  = \sum_{i = 1}^t \big( f(i + N) - f(i)\big)
\end{equation}
and
\begin{equation}
\label{Telescop2}
\sum_{i = 1}^{2N} ( \pm 1)^{i} \big( f(i + 2t) - f(i)\big)  = \sum_{i = 1}^{2t} ( \pm 1)^{i} \big( f(i + 2N) - f(i) \big). 
\end{equation}

Telescoping identities are often used to find sums of finite and infinite Fibonacci and Lucas numbers series in closed 
form \cite{Adegoke,A-Jeannin1997,Farhi,Kilic-Ersanli,Kilic-Prodinger,Melham-JIS,Popov,Rayaguru,Sanford,Zhao}.

\section{New families of reciprocal Horadam series}

Our first main result is the following statement.
\begin{theorem}\label{thm1} 
	Let $m$, $k$, $n$, and $N$ be integers. Then
	\begin{equation}\label{Hor_main1}
	\sum_{i = 1}^N {\frac{{q^{m(i - k)} }}{{w_{m(i - k) + n} w_{m(i + k) + n} }}} = \frac{1}{{e_w u_n u_{2km} }}\sum_{i = 1}^{2k} 
	{\left( \frac{{w_{m(i - k)} }}{w_{m(i - k) + n} }-\frac{{w_{m(i + N - k)} }}{{w_{m(i + N - k) + n} }}  \right)}
	\end{equation}
or, equivalently,
\begin{equation*}\label{Hor_main1--2}
u_{2km} \sum_{i = 1}^N {\frac{{q^{mi} }}{{w_{m(i - k) + n} w_{m(i + k) + n} }}} =
 u_{mN} \sum_{i = 1}^{2k} 
 \frac{q^{mi}}{w_{m(i - k) + n}{w_{m(i + N - k) + n}}}.
\end{equation*}
\end{theorem}
\begin{proof} Writing $n-r$ for $n$ in  identity \eqref{eq.emr6a3s} gives
	\begin{equation*}
	w_{n-r}w_{n+s} - w_{n} w_{n-r+s} = e_w q^{n-r} u_r u_s\,,
	\end{equation*}
	from which, writing $mi-km$ for $n$, $2km$ for $s$ and $-n$ for $r$, we get
	\begin{equation}\label{eq.n5np77t}
	w_{m(i-k)+n}w_{m(i+k)} - w_{m(i-k)}w_{m(i+k)+n} = e_w q^{m(i-k)+n} u_{-n} u_{2km} = -\,e_w q^{m(i-k)} u_{n} u_{2km},
	\end{equation}
	where in the last step we used \eqref{hor_neg}.
	
Now divide through identity~\eqref{eq.n5np77t} by $w_{m(i - k) + n} w_{m(i + k) + n}$ to obtain
	\begin{equation}\label{eq.mo8pkjy}
	\frac{{q^{m(i - k)} }}{{w_{m(i - k) + n} w_{m(i + k) + n} }} = \frac{1}{{ e_wu_n u_{2km} }}
	\left( \frac{{w_{m(i - k)} }}{w_{m(i - k) + n} } -\frac{{w_{m(i + k)} }}{{w_{m(i + k) + n} }}\right).
	\end{equation}

Identify 
$f(i) =\displaystyle \frac{{w_{m(i - k)} }}{{w_{m(i - k) + n} }}$, $t=2k$ and use in the summation formula \eqref{Telescop1} while noting \eqref{eq.mo8pkjy}.
\end{proof}

In particular, evaluation of \eqref{Hor_main1} at $k=1$ and $k=2$ gives
\begin{equation*}
\sum_{i=1}^N \frac{q^{mi}}{w_{m(i-1)+n}w_{m(i+1)+n}} = \frac{q^m}{e_w u_n u_{2m}}
\left( \frac{w_0}{w_{n}} + \frac{w_m}{w_{m+n}} - \frac{w_{m(N+1)}}{w_{m(N+1)+n}} - \frac{w_{mN}}{w_{mN+n}} \right)
\end{equation*}
and
\begin{align*}
\sum_{i=1}^N \frac{q^{mi}}{w_{m(i-2)+n}w_{m(i+2)+n}} & 
= \frac{q^{2m}}{e_w u_n u_{4m}}\left(
\frac{w_{-m}}{w_{-m+n}}+\frac{w_0}{w_{n}}+\frac{w_m}{w_{m+n}} + \frac{w_{2m}}{w_{2m+n}} \right. \nonumber \\
& \quad \left.-\frac{w_{m(N-1)}}{w_{m(N-1)+n}}-\frac{w_{mN}}{w_{mN+n}} - \frac{w_{m(N+1)}}{w_{m(N+1)+n}} - \frac{w_{m(N+2)}}{w_{m(N+2)+n}} 
\right).
\end{align*}

Setting $n=mk$ in Theorem \ref{thm1}, we have the following corollary.
\begin{corollary}\label{maincor1}
	For integers $m$, $k$, and $N$, we have
	\begin{equation*}\label{Hor_main2}
	\sum_{i=1}^N \frac{q^{m(i-k)}}{w_{mi}w_{m(i+2k)}} = \frac{1}{e_w u_{mk} u_{2mk}} \sum_{i = 1}^{2k} 
	\left(\frac{{w_{m(i - k)} }}{{w_{mi} }} - \frac{{w_{m(i + N - k)} }}{{w_{m(i + N)} }} \right)
	\end{equation*}
or, equivalently,
\begin{equation*}
 u_{2mk}\sum_{i=1}^N \frac{q^{mi}}{w_{mi}w_{m(i+2k)}} = u_{mN} \sum_{i = 1}^{2k} {\frac{q^{mi}}{w_{mi}w_{m(i + N)}}}.
\end{equation*}
\end{corollary}

The infinite companion series are evaluated in the next corollary.
\begin{corollary}\label{Hor_inf}
	Let $m$, $k$, and $n$ be integers. Then
	\begin{equation*}\label{Hor_main3_1}
	\sum_{i = 1}^\infty {\frac{{q^{m(i - k)} }}{{w_{m(i - k) + n} w_{m(i + k) + n} }}}  = \frac{1}{{e_w u_n u_{2km} }}
	\left( \sum_{i = 1}^{2k} \frac{{w_{m(i - k)} }}{{w_{m(i - k) + n} }} - \frac{{2k}}{{\alpha ^n }} \right),
	\end{equation*}
	and especially with $n=mk$,
	\begin{equation*}\label{Hor_main3_2}
	\sum_{i = 1}^\infty {\frac{{q^{m(i - k)} }}{{w_{mi} w_{m(i + 2k)} }}}  = \frac{1}{{e_w u_{mk} u_{2km} }}
	\left(\sum_{i = 1}^{2k} \frac{{w_{m(i - k)}}}{w_{mi}} - \frac{{2k}}{\alpha ^{mk}} \right).
	\end{equation*}
\end{corollary}
\begin{proof} 
According to \eqref{Binet}, 
\begin{equation}\label{eq.tsu3992}
\mathop {\lim }_{N \to \infty } \frac{{w_N }}{{w_{N + r} }} = \frac{1}{{\alpha ^r }}.
\end{equation}
Taking limits of both sides of the identity of Theorem \ref{thm1}, making use of \eqref{eq.tsu3992} completes the proof.
\end{proof}

As special cases of our results obtained so far, we have the following Fibonacci and Lucas identities.
\begin{corollary}
	For integers $m,n,N\geq 1$,
	\begin{equation} \label{Fib1}
	\sum_{i=1}^N \frac{(-1)^{m(i-1)}}{F_{m(i-1)+n}F_{m(i+1)+n}} = \frac{1}{F_n F_{2m}} \left( \frac{F_{m(N+1)}}{F_{m(N+1)+n}} 
	+ \frac{F_{mN}}{F_{mN+n}} - \frac{F_m}{F_{m+n}} \right ),
	\end{equation}
	\begin{equation}\label{Luc1}
	\sum_{i=1}^N \frac{(-1)^{m(i-1)}}{L_{m(i-1)+n}L_{m(i+1)+n}} = \frac{1}{5 F_n F_{2m}} \left ( \frac{2}{L_{n}} + \frac{L_m}{L_{m+n}} 
	- \frac{L_{mN}}{L_{mN+n}} - \frac{L_{m(N+1)}}{L_{m(N+1)+n}} \right)
	\end{equation}
and
\begin{equation}\label{Fib2}
	\sum_{i=1}^\infty \frac{(-1)^{m(i-1)}}{F_{m(i-1)+n}F_{m(i+1)+n}} = 
	\frac{1}{F_n F_{2m}} \left ( \frac{2}{\Phi^{n}} - \frac{F_m}{F_{m+n}} \right),
	\end{equation}
	\begin{equation}\label{Luc2}
	\sum_{i=1}^\infty \frac{(-1)^{m(i-1)}}{L_{m(i-1)+n}L_{m(i+1)+n}} = 
	\frac{1}{5 F_n F_{2m}} \left ( \frac{2}{L_{n}} + \frac{L_m}{L_{m+n}} - \frac{2}{\Phi^{n}} \right ),
	\end{equation}
	where $\Phi=(1+\sqrt{5})/2$. 
\end{corollary}
\begin{proof}
	Use Theorem \ref{thm1} and Corollary \ref{Hor_inf} with $w_n=F_n$ and $w_n=L_n$, respectively, and $k=1$. 
	Recall that $e_F=-1$ and $e_L=5$.
\end{proof}

We mention that equations \eqref{Fib1}, \eqref{Luc1}, \eqref{Fib2} and \eqref{Luc2} were discovered 
by the second author recently and appear in \cite[Theorem 1.2]{Frontczak-NNTDM1}. 

Note that from equation \eqref{Hor_main1} it is clear that it does not hold for $m, n, k=0$. The next theorem addresses the situation of $n=0$. 
\begin{theorem}\label{thm2}
	Let $m$, $k$, and $N$ be integers. Then
	\begin{equation}\label{Formula1Th5}
	\sum_{i = 1}^N {\frac{{q^{m(i - k)} }}{{w_{m(i - k)} w_{m(i + k)} }}}  = \frac{1}{{e_w u_{2km}}}\sum_{i = 1}^{2k} 
	{\left( \frac{{w_{m(i + N - k) + 1} }}{{w_{m(i + N - k)} }} - \frac{{w_{m(i - k) + 1} }}{{w_{m(	i - k)} }}\right)}
	\end{equation}
or, equivalently, 
\begin{equation*}
u_{2km}\sum_{i = 1}^N \frac{q^{mi}} {w_{m(i - k)} w_{m(i + k)} }  = u_{mn}\sum_{i = 1}^{2k} \frac{q^{mi}}{w_{m(i - k)}  w_{m(i+N-k)}}.
\end{equation*}
\end{theorem}
\begin{proof}
	Divide through identity \eqref{eq.n5np77t} by $w_{m(i - k)} w_{m(i + k)} $ to obtain
	\[
	\frac{{e_wu_{2km} u_n q^{m(i - k)} }}{{w_{m(i - k)} w_{m(i + k)} }} =  \frac{{w_{m(i + k) + n} }}{{w_{m(i + k)} }} - \frac{{w_{m(i - k) + n} }}{{w_{m(i - k)} }},
	\]
	where $n$ is now arbitrary and can be set equal to unity, yielding
	\begin{equation}\label{eq.dexq31q}
	\frac{{ e_wu_{2km} q^{m(i - k)} }}{{w_{m(i - k)} w_{m(i + k)} }} = 
	\frac{{w_{m(i + k) + 1} }}{{w_{m(i + k)} }}- \frac{{w_{m(i - k) + 1} }}{{w_{m(i - k)} }}\,,
	\end{equation}
	from which the result now follows upon summation over $i$ using \eqref{Telescop1} with
	$\displaystyle 	f(i)=\frac{{w_{m(i - k) + 1} }}{{w_{m(i - k)}}}$.
\end{proof}

Upon letting $n\to+\infty$ in \eqref{Formula1Th5}, we obtain the following.
\begin{corollary}
	Let $m$ and $k$ be integers. Then
	\begin{equation*}
	\sum_{i = 1}^\infty {\frac{{q^{m(i - k)} }}{{w_{m(i - k)} w_{m(i + k)} }}}  = \frac{1}{{e_w u_{2km} }}\left(2k\alpha - 
	\sum_{i = 1}^{2k}{\frac{{w_{m(i - k) + 1} }}{{w_{m(i - k)} }}}\right).
	\end{equation*}
\end{corollary}

Working with Lucas numbers and $k=1$, we immediately get the next results:
	\begin{equation*}\label{v1_L}
	\sum_{i=1}^N \frac{(-1)^{m(i-1)}}{L_{m(i-1)}L_{m(i+1)}} = \frac{1}{5 F_{2m}}\left ( \frac{L_{m(N+1)+1}}{L_{m(N+1)}} 
	+ \frac{L_{mN+1}}{L_{mN}} - \frac{L_{m+1}}{L_{m}} - \frac{1}{2} \right )
	\end{equation*}
	and
	\begin{equation*}\label{v2_L}
	\sum_{i=1}^\infty \frac{(-1)^{m(i-1)}}{L_{m(i-1)}L_{m(i+1)}} = \frac{1}{5F_{2m}}\left (2\Phi - \frac{L_{m+1}}{L_{m}} - \frac{1}{2}\right ).
	\end{equation*}
The above Lucas sums are also evaluated in \cite{Frontczak-NNTDM1}. The results, however, are stated in a different form as follows
\begin{equation*}\label{v1_L_Fr}
\sum_{i=1}^N \frac{(-1)^{m(i-1)}}{L_{m(i-1)}L_{m(i+1)}} = \frac{1}{2 F_{2m}}\left ( \frac{F_{m(N+1)}}{L_{m(N+1)}} 
+ \frac{F_{mN}}{L_{mN}} - \frac{F_m}{L_{m}} \right )
\end{equation*}
and
\begin{equation*}\label{v2_L_Fr}
\sum_{i=1}^\infty \frac{(-1)^{m(i-1)}}{L_{m(i-1)}L_{m(i+1)}} = \frac{1}{\sqrt{5}F_{2m}} - \frac{1}{2L^2_{m}}.
\end{equation*}
The reason for the differences in expressing these sums is, that the special case of Theorem \ref{thm2} with $w_n=v_n$ possesses a different expression. 

As the family of series of Lucas numbers of the second kind is interesting on its own we give this expression in a separate theorem.
\begin{theorem}\label{thm3}
	For integers $m$ and $k$, we have the following identities
	\begin{align*}
	\sum_{i = 1}^N {\frac{{q^{m(i - k)} }}{{v_{m(i - k)} v_{m(i + k)} }}}&	= 
	\frac{1}{{2u_{2km} }}\sum_{i = 1}^{2k} {\left( {\frac{{u_{m(i + N - k)} }}{{v_{m(i + N - k)} }} - \frac{{u_{m(i - k)} }}{{v_{m(i - k)} }}} \right)}\notag\\
	&=\frac{u_{m N}}{u_{2km}}\sum_{i = 1}^{2k} \frac{q^{m(i-k)} }{{v_{m(i - k)}v_{m(i +N- k)}}}
	\end{align*}
	and
	\begin{align*}
	\sum_{i = 1}^N {\frac{{q^{mi} }}{{u_{mi} u_{m(i + 2k)} }}}
	 &= \frac{1}{{2u_{2km} }}\sum_{i = 1}^{2k} \left( {\frac{{v_{mi} }}{{u_{mi} }} - \frac{{v_{m(i + N)} }}{u_{m(i + N)} }} \right)\notag\\
	 &=\frac{u_{mN}}{u_{2km}}\sum_{i = 1}^{2k} \frac{q^{mi}} {u_{mi}u_{m(i + N)}}.
	\end{align*}
\end{theorem}

\begin{proof} 
The proof is similar to that of Theorem \ref{thm1}. Here we use
	\begin{equation}\label{temp}
	u_{m(i + k)} v_{m(i - k)}  - u_{m(i - k)} v_{m(i + k)}  = 2q^{m(i - k)} u_{2mk},
	\end{equation}
	which is obtained by setting $s=m(i+k)$ and $t=m(i-k)$ in the identity
	\begin{equation*}\label{eq.y4wx6me}
	u_s v_t - v_s u_t = -2q^s u_{t-s}.
	\end{equation*}
	Dividing through identity \eqref{temp} by $v_{m(i-k)}v_{m(i+k)}$ gives
	\begin{equation*}\label{eq.qs03pax}
	\frac{{2u_{2km} q^{m(i - k)} }}{{v_{m(i - k)} v_{m(i + k)} }} = \frac{{u_{m(i + k)} }}{{v_{m(i + k)} }} - \frac{{u_{m(i - k)} }}{{v_{m(i - k)} }},
	\end{equation*}
	while dividing through by $u_{m(i-k)}u_{m(i+k)}$ and shifting the index $i$ gives
	\begin{equation*}\label{eq.lm6jz6w}
	\frac{{2u_{2km} q^{mi} }}{{u_{mi} u_{m(i + 2k)} }} = \frac{{v_{mi} }}{{u_{mi} }} - \frac{{v_{m(i + 2k)} }}{{u_{m(i + 2k)} }}.
	\end{equation*}
	This completes the proof of the second identity.
\end{proof}

As a by-product from Theorems \ref{thm2} and \ref{thm3} we obtain the following relation involving Lucas sequences of the first 
and second kind: 
\begin{equation*}
\sum_{i = 1}^{2k} {\left( \frac{u_{m(i + N - k)}}{v_{m(i + N - k)}} - \frac{u_{m(i - k)}}{v_{m(i - k)}} \right)}
=\frac{2}{\Delta} \sum_{i = 1}^{2k}{\left( \frac{v_{m(i + N - k) + 1}}{v_{m(i + N - k)} } - \frac{v_{m(i - k) + 1}}{v_{m(i - k)}}\right)}.
\end{equation*}
Similarly, comparing the second part of Theorem \ref{thm3} with Corollary \ref{maincor1} ($w_n=u_n$) we get the interesting relation
\begin{equation*}
\sum_{i = 1}^{2k} {\left(\frac{v_{mi}}{u_{mi}}-\frac{v_{m(i+N)}}{u_{m(i + N)}} \right)}
=\frac{2q^{mk}}{u_{mk}} \sum_{i = 1}^{2k}{\left( \frac{u_{m(i + N - k) + 1}}{u_{m(i + N)} } - \frac{u_{m(i - k)}}{u_{mi}}\right)}.
\end{equation*}
Note that in the above sum relations $m$ and $N$ are fixed parameters while $k$ varies.

We conclude this section with the observation, that the identity of Theorem~\ref{thm2} will crash in general for sequences with $w_0=a=0$, such as the Lucas sequence of the first kind. 
We now give a non-singular version of the theorem. The proof follows similar to that of Theorem~\ref{thm1} and hence it is omitted.
\begin{theorem}
	Let $m$, $k$, and $N$ be integers. Then
	\begin{align*}
	\sum_{i = 1}^N \frac{q^{mi} }{w_{mi} w_{m(i + 2k)} } &= \frac{1}{e_w u_{2km}} \sum_{i = 1}^{2k} 
	{\left( \frac{{w_{m(i + N) + 1} }}{{w_{m(i + N)} }} - \frac{{w_{mi + 1} }}{{w_{mi} }} \right)}\notag\\
	&= \frac{u_{mN}}{u_{2km}} \sum_{i = 1}^{2k} 
	 \frac{q^{mi}}{w_{mi}w_{m(i + N)}}, 
	\end{align*}
	as well as
	\begin{align*}
	\sum_{i = 1}^\infty \frac{{q^{mi} }}{{w_{mi} w_{m(i + 2k)} }} = \frac{1}{e_w u_{2km} }\left(2k\alpha - \sum_{i = 1}^{2k} 
	{\frac{{w_{mi + 1} }}{w_{mi} }} \right).
	\end{align*}
\end{theorem}

\section{Still other Horadam series}

The next achievement of this paper is the following formula. 
 \begin{theorem}
 	Let $m$, $k$, $n$, and $N$ be integers. Then
 	\begin{equation}\label{Hor_main4}
 	\sum_{i = 1}^{2N} {\frac{{(\pm1)^{i} q^{m(i - k)} }}{{w_{m(i - k) + n} w_{m(i + k) + n} }}} = \frac{1}{{e_w u_n u_{2km} }}
 	\sum_{i = 1}^{2k} { (\pm1)^{i} \left(  \frac{{w_{m(i - k)} }}
 				{{w_{m(i - k) + n}}} - \frac{{w_{m(i + 2N - k)} }}{{w_{m(i + 2N - k) + n} }}\right) } 
 	\end{equation}
or, equivalently,
 \begin{equation*}
u_{2km} \sum_{i = 1}^{2N} {\frac{{(\pm1)^{i} q^{mi} }}{{w_{m(i - k) + n} w_{m(i + k) + n} }}} = u_{2mN}
 \sum_{i=1}^{2k}{\frac{(\pm1)^{i}q^{mi} }{w_{m(i - k) + n}w_{m(i + 2N - k) + n}}}. 
 \end{equation*}
\end{theorem}
 \begin{proof}
 	Use $\displaystyle f(i) = \frac{{w_{m(i - k)} }}{{w_{m(i - k) + n} }}$ and $t=k$ in 
 \eqref{Telescop2} to obtain, by~\eqref{eq.mo8pkjy},
 	\[
 	f(i + 2k) - f(i) = \frac{{w_{m(i + k)} }}{{w_{m(i + k) + n} }} - \frac{{w_{m(i - k)} }}{{w_{m(i - k) + n} }} = -\frac{{e_wu_n u_{2km} q^{m(i - k)} }}{{w_{m(i - k) + n} w_{m(i + k) + n} }},
 	\]
 	and
 	\[
 	f(i + 2N) - f(i) = \frac{{w_{m(i + 2N - k)} }}{{w_{m(i + 2N - k) + n} }} - \frac{{w_{m(i - k)} }}{{w_{m(i - k) + n} }}.
 	\]
 	
 	Putting these values into the summation formula~\eqref{Telescop2} produces the stated identity \eqref{Hor_main4}.
 \end{proof}
 
 Letting $N$ to $+\infty$, we immediate  obtain from \eqref{Hor_main4} the following corollary.
 \begin{corollary}
 	Let $m$, $k$, and $n$ be integers. Then
 	\[
 	\sum_{i = 1}^\infty {\frac{{( - 1)^{i} q^{m(i - k)} }}{{w_{m(i - k) + n} w_{m(i + k) + n} }}} = \frac{1}{ e_w u_n u_{2km}}\sum_{i = 1}^{2k} {\frac{{( - 1)^{i}w_{m(i - k)} }}{{w_{m(i-k) + n} }}}. 
 	\]
 \end{corollary}
 \begin{theorem}\label{thm6} 
 	Let $m$, $k$, $n$, and $N$ be integers. Then
 	\[
 	\sum_{i = 1}^N {\frac{{(\pm1)^iq^{m(2i - k)} }}{{w_{m(2i - k) + n} w_{m(2i + k) + n} }}} = \frac{1}{{e_w u_n u_{2km} }}\sum_{i = 1}^{k}(\pm1)^i 
 	{\left( \frac{{w_{m(2i - k)} }}{{w_{m(2i - k) + n} }} - \frac{{w_{m(2(i + N) - k)} }}{{w_{m(2(i + N) - k) + n} }}\right)}
 	\]
or, equivalently, 
\[
u_{2km}\sum_{i = 1}^N {\frac{(\pm1)^iq^{2mi}}{{w_{m(2i - k) + n} w_{m(2i + k) + n} }}} = u_{2mN} \sum_{i = 1}^{k} 
\frac{(\pm1)^iq^{2mi}}{w_{m(2i - k) + n} w_{m(2(i + N) - k)+n}}.
\] 
\end{theorem}
 \begin{proof}
 	Write $2i$ for $i$ in \eqref{eq.mo8pkjy} to obtain
 	\begin{equation*}\label{eq.exvg74n}
 	\frac{{q^{m(2i - k)} }}{{w_{m(2i - k) + n} w_{m(2i + k) + n} }} = \frac{1}{{ e_wu_n u_{2km} }}\left( \frac{{w_{m(2i - k)} }}{{w_{m(2i - k) + n}}} - \frac{{w_{m(2i + k)} }}{{w_{m(2i + k) + n} }} \right).
 	\end{equation*}
 	Use $\displaystyle f(i)=\frac{{w_{m(2i - k)} }}{{w_{m(2i - k) + n}}}$ and $t=k$ in 
 	\eqref{Telescop1}.
 \end{proof}
 
 In the limit as $N$ approaches infinity in Theorem \ref{thm6}, we have  the following infinite series.
 \begin{corollary}\label{corollary.l7biie2} 
 	Let $m$, $k$, and $n$ be integers. Then
 	\[
 	\sum_{i = 1}^\infty {\frac{{q^{m(2i - k)} }}{{w_{m(2i - k) + n} w_{m(2i + k) + n} }}} = \frac{1}{{e_w u_n u_{2km} }}\left(\sum_{i = 1}^{k} 
 	{ {\frac{{w_{m(2i - k)} }}{{w_{m(2i - k) + n} }}}} - \frac k{\alpha^n}\right)
 	\]
 and
  	\[
 \sum_{i = 1}^\infty {\frac{(- 1)^{i}{q^{m(2i - k)} }}{{w_{m(2i - k) + n} w_{m(2i + k) + n} }}} = \frac{1}{{e_w u_n u_{2km} }}\sum_{i = 1}^{k} 
 {\frac{{(-1)^{i}w_{m(2i - k)} }}{{w_{m(2i - k) + n} }}} \,.
 \]
 \end{corollary}
  
 Now we list some Fibonacci and Lucas series which follow form Corollary \ref{corollary.l7biie2}:
	\[
\sum_{i = 1}^\infty {\frac{1}{{F_{m(2i - 1) + n} F_{m(2i + 1) + n} }}} = \frac{(-1)^m}{F_n F_{2m}}\left( 
\frac{1}{\Phi^n} -\frac{F_{m}}{F_{m + n}} \right),
\]
\[
\sum_{i = 1}^\infty {\frac{1}{{F_{2m(i - 1) + n} F_{2m(i + 1) + n} }}} = \frac{1}{F_n F_{4m}}\left(\frac 2{\Phi^n}- \frac{F_{2m}}{F_{2m + n}}\right),
\]
\[
\sum_{i = 1}^\infty {\frac{1}{{L_{m(2i - 1) + n} L_{m(2i + 1) + n} }}} = \frac{(-1)^{m}}{{5F_n F_{2m} }}\left(
{ \frac{{L_m }}{{L_{m + n} }}} - \frac 1{\Phi^n}\right),
\]
\[
\sum_{i = 1}^\infty \frac{1}{L_{2m(i - 1) + n} L_{2m(i +1) + n}} = \frac{1}{5F_nF_{4m}}\left( 
{\frac{2}{L_{n}}}+\frac{{L_{2m} }}{L_{2m+ n} }-\frac{2}{\Phi^n}\right),
\]
 and
 \[
 \sum_{i = 1}^\infty {\frac{(- 1)^{i}}{{F_{m(2i - 1) + n} F_{m(2i + 1) + n} }}} =  \frac{(-1)^mF_{m}}{F_n F_{2m}F_{m + n} },
 \]
 \[
 \sum_{i = 1}^\infty {\frac{(- 1)^i}{{F_{2m(i - 1) + n} F_{2m(i + 1) + n} }}} = -\frac{F_{2m}}{{F_n F_{4m}F_{2m + n}}},
 \]
 \[
 \sum_{i = 1}^\infty {\frac{(- 1)^{i}}{{L_{m(2i - 1) + n} L_{m(2i + 1) + n} }}} = \frac{(-1)^{m-1}L_m}{5 F_n F_{2m}L_{m + n} },
 \]
 \[
 \sum_{i = 1}^\infty \frac{(- 1)^i}{L_{2m(i - 1) + n} L_{2m(i + 1) + n} } = \frac{1}{5 F_n F_{4m} }
 \left(  \frac{L_{2m}}{L_{2m + n} } -\frac{2}{L_n}  \right).
 \]
 
 The next theorem is a non-singular version of the first identity from Theorem \ref{thm6} and Corollary \ref{corollary.l7biie2} in case $n=0$.
 \begin{theorem}
 	Let $m$, $k$, and $N$ be integers. Then
 	\begin{align*}
 	\sum_{i = 1}^N {\frac{(\pm1)^iq^{m(2i - k) }}{{w_{m(2i - k)} w_{m(2i + k)} }}} & = \frac{1}{e_wu_{2km}}\sum_{i = 1}^{k} (\pm1)^i{\left( \frac{{w_{m(2(i + N) - k) + 1} }}{{w_{m(2(i + N) - k)}}}-\frac{{w_{m(2i - k) + 1} }}{{w_{m(2i - k)} }} \right)}\\
 	&=\frac{u_{2mN}}{u_{2km}} \sum_{i = 1}^{k} \frac{(\pm1)^iq^{m(2i - k) }} {w_{m(2i - k)}w_{m(2(i+N)-k)}}
 	\end{align*}
 and
 \begin{equation}
 \sum_{i = 1}^\infty {\frac{{q^{m(2i - k)} }}{{w_{m(2i - k)} w_{m(2i + k)} }}}  = \frac{1}{{ e_wu_{2km} }}\left({k\alpha - \sum_{i = 1}^{k} {\frac{{w_{m(2i - k) + 1} }}{{w_{m(2i - k)} }} } }\right),\label{infin1}
 \end{equation}
 \begin{equation}
 \sum_{i=1}^\infty{\frac{(-1)^{i}{q^{m(2i-k)}}}{w_{m(2i-k)} w_{m(2i + k)}}} = \frac{1}{{e_wu_{2km} }}\sum_{i = 1}^{k} {( - 1)^{i-1}\frac{{w_{m(2i - k) + 1} }}{w_{m(2i-k)}}}. \label{infin2}
 \end{equation}
\end{theorem}
 \begin{proof} Write $2i$ for $i$ in identity~\eqref{eq.dexq31q} to obtain
 	\[
 	-\frac{{e_wu_{2km} q^{m(2i - k)} }}{{w_{m(2i - k)} w_{m(2i + k)} }} = \frac{{w_{m(2i - k) + 1} }}{{w_{m(2i - k)} }} - \frac{{w_{m(2i + k) + 1} }}{{w_{m(2i + k)} }}\,,
 	\]
from which the result now follows upon summation over $i$ using \eqref{Telescop1} with $\displaystyle f(i)=\frac{{w_{m(2i - k) + 1} }}{{w_{m(2i - k)} }}$ and $t=k$. Taking limit as $N\to\infty$, we obtain \eqref{infin1} and \eqref{infin2}. 
\end{proof}

\section{Conclusion}

We have evaluated some new three parameter families of reciprocal Horadam sums in closed form.  The approach is elementary and is based on clever telescoping. It seems possible to extend the results of the present article to reciprocal sums involving three and four Horadam numbers 
as factors in the denominator. This could be explored further in a future project.

\bigskip
\hrule
\bigskip

\noindent 2010 {\it Mathematics Subject Classification}: Primary 11B37; Secondary, 11B39.

\noindent \emph{Keywords:} Horadam sequence, Fibonacci sequence, Lucas sequence, reciprocal sum.

\bigskip
\hrule
\bigskip

\noindent (Concerned with sequences
\seqnum{A000032} and \seqnum{A000045}.)


\bigskip
\hrule
\bigskip


\end{document}